\documentclass[11pt, a4paper]{amsart}
\usepackage{fullpage}
\usepackage{amssymb}
\usepackage{commath}
\usepackage{mathrsfs}
\usepackage{empheq}
\usepackage{amsmath, amsthm}
\usepackage{multicol}
\usepackage{parskip}
\usepackage{graphicx}
\usepackage{realboxes}
\usepackage{comment}

\usepackage[colorlinks=true, allcolors=blue]{hyperref}

\usepackage{tikz,tikz-3dplot,tikz-cd,tkz-tab,tkz-euclide,pgf,pgfplots}
\usetikzlibrary{bending,quotes,angles,3d}
\usepackage{bbm}
\usepackage{url}

\newcommand{\ol}{\overline}
\newcommand{\ul}{\underline}
\renewcommand{\phi}{\varphi}
\newcommand{\llb}{\left\lbrace}
\newcommand{\rrb}{\right\rbrace}
\newcommand{\llv}{\left\lvert}
\newcommand{\rrv}{\right\rvert}
\newcommand{\1}{\mathbbm{1}}

\newcommand\fakeqed{\pushQED{\qed}\qedhere}

\DeclareMathOperator{\Tor}{Tor}
\DeclareMathOperator{\Ext}{Ext}
\DeclareMathOperator{\Hom}{Hom}
\DeclareMathOperator{\cupmod}{Cup}

\newtheorem{thm}{Theorem}[section]

\theoremstyle{plain}

\newtheorem{lem}[thm]{Lemma}

\newtheorem*{theorem*}{Theorem}
\newtheorem*{cor*}{Corollary}

\theoremstyle{definition}
\newtheorem{defn}[thm]{Definition}

\theoremstyle{remark}
\newtheorem{rem}[thm]{Remark}
\newtheorem{eg}[thm]{Example}

\makeatletter
\@namedef{subjclassname@1991}{\textup{2020} Mathematics Subject Classification}
\makeatother

\title{Cohomology of dilute Temperley--Lieb algebras}
\author{Andrew Fisher}
\address{(Andrew Fisher) School of Mechanical, Aerospace and Civil Engineering, University of Sheffield, Sir Frederick Mappin Building, Mappin Street, Sheffield, S1 4DT, UK}
\email{andrew.fisher@sheffield.ac.uk}
\author{Daniel Graves}
\address{(Daniel Graves) Lifelong Learning Centre, University of Leeds, Woodhouse, Leeds, LS2 9JT, UK}
\email{dan.graves92@gmail.com}
\date{}

\begin{document}



\begin{abstract}
Dilute Temperley--Lieb algebras are variants of Temperley--Lieb algebras arising in statistical mechanics in the study of solvable lattice models. In this paper we prove that the (co)homology of dilute Temperley--Lieb algebras vanishes in all positive degrees.
\end{abstract}

\maketitle

\section{Introduction}

The dilute Temperley--Lieb algebras have their origins in statistical mechanics. They arise in the study of solvable lattice models, in particular, the dilute $A$-$D$-$E$ lattice models (see \cite{Nienhuis,Roche,WNS1,WNS2,WPSN,Grimm,BP,GN1,GN2} for instance). They are variants of the Temperley--Lieb algebras, first introduced in \cite{TL}. The representation theory of the dilute Temperley--Lieb algebras has been studied in \cite{BSA,B-fusion}. The authors suggest \cite{BSA} for a good introduction to and history of the dilute Temperley--Lieb algebras. 

Both Temperley--Lieb algebras and dilute Temperley--Lieb algebras are defined on bases of diagrams, a viewpoint originally due to Kauffman \cite{Kauffmann1}. For a positive integer $n$, a \emph{Temperley--Lieb $n$-diagram} is a graph on two identical columns of $n$ vertices placed side by side such that every vertex is connected to \emph{precisely} one other vertex by an edge, and the edges can be drawn without crossing within the rectangle formed by the vertices. We refer to such a graph as being \emph{planar}. \emph{Dilute Temperley--Lieb $n$-diagrams} are defined similarly except we now allow isolated vertices. 

If we now fix a commutative ring, $k$, and an element $\delta \in k$, the \emph{Temperley--Lieb algebra} $TL_n(\delta)$ is spanned linearly by the Temperley--Lieb $n$-diagrams. The product glues two diagrams together along a column of vertices and generates a factor of $\delta$ for every loop formed. Similarly, the \emph{dilute Temperley--Lieb algebra}, $dTL_n(\delta)$, is spanned $k$-linearly by dilute Temperley--Lieb $n$-diagrams. The product in the dilute Temperley--Lieb algebra is similar to the product in the Temperley--Lieb algebra, insofar as it is defined in terms of gluing two basis diagrams together. However, we now have extra conditions. Loosely speaking, if we glue two diagrams together and there is a \emph{floating edge} in the diagram, then the product is defined to be zero (see Definition \ref{dTL-alg-defn} below for the precise definition and examples).

The (co)homology of Temperley--Lieb algebras has been much studied in recent years. The vanishing of the (co)homology of the Temperley--Lieb algebra has been shown to vary depending on the invertibility of $\delta$ (see \cite[Theorem A]{BH1}, \cite[Theorem C]{Sroka} and \cite[Theorem 7.6]{Boyde}) and the parity of $n$ (see \cite[Theorem B]{BH1}, \cite[Theorem A]{Sroka} and \cite[Theorem 1.1]{Boyde}). It is known that the (co)homology of $TL_n(\delta)$ can be non-zero in higher degrees when $n$ is even \cite[Theorem C]{BH1}. More recently, it has been shown that the homology of even Temperley--Lieb algebras is highly non-trivial, being governed by the differential graded algebra of planar loops \cite{BBRWS,Boyde-planar}. 

We show that the (co)homology of the dilute Temperley--Lieb algebras demonstrates different behaviour to that of the Temperley--Lieb algebras. Intuitively, the reason for this is that the dilute Temperley--Lieb algebra contains more idempotents, arising from the isolated vertices which are permitted in the dilute Temperley--Lieb algebras but not in the Temperley--Lieb algebras. This is made precise in Section \ref{main-thm-sec}. Our main result is as follows.

\begin{thm}
\label{main-thm}
For any commutative ring $k$, any $\delta \in k$ and any $n\in \mathbb{N}$, the
(co)homology of $dTL_n(\delta)$ is concentrated in degree zero, where it is isomorphic to $k$.
\end{thm}

Our method of proof will use material from Boyde \cite{Boyde2}, which is a generalization of Sroka's work. In Section \ref{alg-sec}, we observe that the dilute Temperley--Lieb algebras can be equipped with an augmentation sending one particular basis diagram to $1\in k$ and all others to $0\in k$. There exists a two-sided ideal $I$ of $dTL_n(\delta)$ spanned $k$-linearly by all basis diagrams apart from this particular diagram and we note that $dTL_n(\delta)/I\cong k$. We identify a family of left ideals of $dTL_n(\delta)$ which cover the ideal $I$ and which allows us to form a chain complex called the \emph{Mayer--Vietoris complex} (recalled in Section \ref{Mayer--Vietoris-sec}). This is a generalization of Sroka's \emph{cellular Davis complex} \cite[Definition 8]{Sroka}. We show that, in our case, this complex is a projective resolution of the trivial $dTL_n(\delta)$-module, which we denote by $\1$, by left $dTL_n(\delta)$-modules. Finally we show that the functors $\1\otimes_{dTL_n(\delta)} -$ and $\Hom_{dTL_n(\delta)}(-,\1)$ are zero when evaluated on the positive degrees of our resolution, from which our result follows.

\subsection*{Conventions}
Throughout, $k$ will be a unital, commutative ring and $n$ will be a positive integer unless otherwise stated. We will denote the set $\lbrace 1,2,\dotsc, n\rbrace$ by $\ul{n}$.

\subsection*{Acknowledgements}
This work was carried out whilst the first author was a PhD student in the School of Mathematical and Physical Sciences at the University of Sheffield. We would like to thank James Brotherston and Natasha Cowley for helpful and interesting conversations whilst writing this paper. We would like to thank James Cranch and Sarah Whitehouse for their feedback and support in this project and related work. We would like to thank Rachael Boyd and Richard Hepworth for interesting conversations at the 2024 British Topology Meeting in Aberdeen. The second author is grateful to Guy Boyde for related conversations at the Isaac Newton Institute for Mathematical Sciences, Cambridge during the programme \emph{Equivariant homotopy theory in context} supported by EPSRC grant EP/Z000580/1. We would like to thank the anonymous referee. Their comments have led to our main result being significantly strengthened and the presentation of the paper being improved.

\section{Dilute Temperley--Lieb algebras}
\label{alg-sec}
In this section we give a precise definition of the dilute Temperley--Lieb algebras.

\begin{defn}
A \emph{dilute Temperley--Lieb $n$-diagram} consists of two identical columns of $n$ vertices placed side by side such that each vertex is connected to at most one other by an edge and edges are planar, in the sense that they can be drawn without crossings within the rectangle formed by the vertices.
\end{defn}

The vertices down the left-hand column of a dilute Temperley--Lieb $n$-diagram will be labelled by $1,\dots , n$ in ascending order from top to bottom and the vertices down the right-hand column will be labelled by $\ol{1},\dotsc , \ol{n}$ in ascending order from top to bottom. We identify dilute Temperley--Lieb diagrams up to isotopy.

\begin{defn}
\label{terminology-defn}
An edge that connects the left-hand column of vertices to the right-hand column of vertices will be called a \emph{propagating edge}. An edge that connects two vertices in the same column will be called a \emph{non-propagating edge}. A vertex not connected to any other by an edge will be called an \emph{isolated vertex}.
\end{defn}

With our basis diagrams in place we can make the definition of the dilute Temperley--Lieb algebras precise.

\begin{defn}
\label{dTL-alg-defn}
Let $\delta \in k$. The \emph{dilute Temperley--Lieb algebra}, $dTL_n(\delta)$, is the $k$-algebra with basis consisting of all dilute Temperley--Lieb $n$-diagrams with the multiplication defined by the $k$-linear extension of the following product of diagrams. Let $d_1$ and $d_2$ be dilute Temperley--Lieb $n$-diagrams. The product $d_1d_2$ is obtained by the following procedure:
\begin{itemize}
    \item Place the diagram $d_2$ to the right of the diagram $d_1$ and identify the vertices $\ol{1},\dotsc , \ol{n}$ in $d_1$ with the vertices $1,\dotsc , n$ in $d_2$. Call this diagram with three columns of vertices $d_1\ast d_2$. We drop the labels of the vertices in the middle column and we preserve the labels of the left-hand column and right-hand column.
    \item If we have an edge which is connected to neither the left-hand column nor the right-hand column and is not part of a loop, then the product is zero. Similarly, if we have an edge that connects the middle column to either the left-hand column or the right-hand column, but not both, then the product is zero. Such an edge is sometimes called a \emph{floating edge}.  
    \item Otherwise, count the number of loops that lie entirely within the middle column. Call this number $\alpha$.
    \item Make a new dilute Temperley--Lieb $n$-diagram, $d_3$, as follows. Given distinct vertices $x$ and $y$ in the set $\llb 1, \ol{1}, \dotsc , n , \ol{n}\rrb$, $d_3$ has an edge between $x$ and $y$ if there is a path from $x$ to $y$ in $d_1\ast d_2$. One can view the diagram $d_3$ as the result of discarding the vertices and loops in the middle column.
    \item We define $d_1d_2=\delta^{\alpha}d_3$.
\end{itemize}
The identity element is the formal sum of all diagrams that can be formed from the unique diagram with $n$ propagating edges by omitting $k$ edges ($0\leqslant k \leqslant n$).
\end{defn}

As noted in \cite[Subsection 2.1]{BSA}, this product is associative since the formation of propagating edges, closed loops and floating edges in a composite of three or more diagrams is independent of the order of multiplication.

\begin{eg}
The identity element in $dTL_2(\delta)$ is 
\begin{center}
\begin{tikzpicture}
\tikzset{>=stealth}
\draw (0,0) node[left] {\footnotesize $2$};
\draw (0,1) node[left] {\footnotesize $1$};
\draw (1,0) node[right] {\footnotesize $\ol{2}$};
\draw (1,1) node[right] {\footnotesize $\ol{1}$};
\fill (0,1) circle[radius=2pt];
\fill (0,0) circle[radius=2pt];
\fill (1,1) circle[radius=2pt];
\fill (1,0) circle[radius=2pt];
\draw (0,0) -- (1,0);
\draw (0,1) -- (1,1);
\draw (2,0.5) node {$+$};
\draw (3,0) node[left] {\footnotesize $2$};
\draw (3,1) node[left] {\footnotesize $1$};
\draw (4,0) node[right] {\footnotesize $\ol{2}$};
\draw (4,1) node[right] {\footnotesize $\ol{1}$};
\fill (3,1) circle[radius=2pt];
\fill (3,0) circle[radius=2pt];
\fill (4,1) circle[radius=2pt];
\fill (4,0) circle[radius=2pt];
\draw (3,1) -- (4,1);
\draw (5,0.5) node {$+$};
\draw (6,0) node[left] {\footnotesize $2$};
\draw (6,1) node[left] {\footnotesize $1$};
\draw (7,0) node[right] {\footnotesize $\ol{2}$};
\draw (7,1) node[right] {\footnotesize $\ol{1}$};
\fill (6,1) circle[radius=2pt];
\fill (6,0) circle[radius=2pt];
\fill (7,1) circle[radius=2pt];
\fill (7,0) circle[radius=2pt];
\draw (6,0) -- (7,0);
\draw (8,0.5) node {$+$};
\draw (9,0) node[left] {\footnotesize $2$};
\draw (9,1) node[left] {\footnotesize $1$};
\draw (10,0) node[right] {\footnotesize $\ol{2}$};
\draw (10,1) node[right] {\footnotesize $\ol{1}$};
\fill (9,1) circle[radius=2pt];
\fill (9,0) circle[radius=2pt];
\fill (10,1) circle[radius=2pt];
\fill (10,0) circle[radius=2pt];
\end{tikzpicture}    
\end{center}
\end{eg}

\begin{eg}
Here are some examples of composition in the dilute Temperley--Lieb algebra $dTL_2(\delta)$. We drop the labels on the vertices for a clearer picture. The first composite below is zero because we have a floating edge connecting the left-hand column to the middle column.
\begin{center}
\begin{tikzpicture}
\fill (0,0) circle[radius=2pt];
\fill (1,0) circle[radius=2pt];
\fill (0,1) circle[radius=2pt];
\fill (1,1) circle[radius=2pt]; 
\fill (2,0) circle[radius=2pt];
\fill (2,1) circle[radius=2pt]; 
\draw (0,1) -- (1,1) -- (2,1);
\draw (0,0) -- (1,0);
\draw (2.5,0.5) node {$=$};
\draw (3,0.5) node {$0$};
\fill (6,0) circle[radius=2pt];
\fill (7,0) circle[radius=2pt];
\fill (6,1) circle[radius=2pt];
\fill (7,1) circle[radius=2pt]; 
\fill (8,0) circle[radius=2pt];
\fill (8,1) circle[radius=2pt];
\draw (6,0) -- (7,1) -- (8,0);
\draw (8.5,0.5) node {$=$};
\fill (9,0) circle[radius=2pt];
\fill (9,1) circle[radius=2pt];
\fill (10,0) circle[radius=2pt];
\fill (10,1) circle[radius=2pt]; 
\draw (9,0) -- (10,0);
\end{tikzpicture}
\end{center}
The first composite below is zero because we have an edge which is contained entirely within the middle column. In the second composite below we form a loop in the middle column and so we obtain a factor of $\delta \in k$.
\begin{center}
\begin{tikzpicture}
\fill (0,0) circle[radius=2pt];
\fill (1,0) circle[radius=2pt];
\fill (0,1) circle[radius=2pt];
\fill (1,1) circle[radius=2pt]; 
\fill (2,0) circle[radius=2pt];
\fill (2,1) circle[radius=2pt];
\path[-] (0,0) edge [bend right=20] (0,1);
\path[-] (1,0) edge [bend left=20] (1,1);
\path[-] (2,0) edge [bend left=20] (2,1);
\draw (2.5,0.5) node {$=$};
\draw (3,0.5) node {$0$};
\fill (6,0) circle[radius=2pt];
\fill (7,0) circle[radius=2pt];
\fill (6,1) circle[radius=2pt];
\fill (7,1) circle[radius=2pt]; 
\fill (8,0) circle[radius=2pt];
\fill (8,1) circle[radius=2pt];
\path[-] (6,0) edge [bend right=20] (6,1);
\path[-] (7,0) edge [bend left=20] (7,1);
\path[-] (7,0) edge [bend right=20] (7,1);
\path[-] (8,0) edge [bend left=20] (8,1);
\draw (8.5,0.5) node {$=$};
\draw (9,0.5) node {$\delta$};
\fill (9.5,0) circle[radius=2pt];
\fill (10.5,0) circle[radius=2pt];
\fill (9.5,1) circle[radius=2pt];
\fill (10.5,1) circle[radius=2pt];
\path[-] (9.5,0) edge [bend right=20] (9.5,1);
\path[-] (10.5,0) edge [bend left=20] (10.5,1);
\end{tikzpicture}    
\end{center}
We refer the reader to \cite{BSA} for further examples of composition in dilute Temperley--Lieb algebras.
\end{eg}

\begin{rem}
The basis diagrams of $dTL_n(\delta)$ are the same as the basis diagrams of the Motzkin algebra (see \cite{BHal,JoYa} for instance). However, the product of diagrams in the Motzkin algebra is different to the product of diagrams in the dilute Temperley--Lieb algebra.    
\end{rem}

Since the composite of two dilute Temperley--Lieb $n$-diagrams with $i$ propagating edges and $j$ propagating edges respectively is a scalar multiple of a diagram with at most $\min(i,j)$ propagating edges (see \cite[Subsection 2.1]{BSA} for example), we can define a two-sided ideal $I$ of $dTL_n(\delta)$ spanned $k$-linearly by dilute Temperley--Lieb $n$-diagrams having at most $n-1$ propagating edges. We therefore have an augmentation $\varepsilon \colon dTL_n(\delta) \rightarrow k$ determined by sending the unique diagram with $n$ propagating edges to $1\in k$ and all other diagrams to $0\in k$. We denote by $\mathbbm{1}$ the \emph{trivial $dTL_n(\delta)$-module}: a single copy of the ground ring $k$, where $dTL_n(\delta)$ acts on $k$ via this augmentation. This allows us to define the homology and cohomology of the dilute Temperley--Lieb algebras as $\Tor_{\star}^{dTL_n(\delta)}(\1,\1)$ and $\Ext_{dTL_n(\delta)}^{\star}(\1,\1)$ following \cite[Definition 2.4.4]{Benson1}.

\section{Mayer--Vietoris complex}
\label{Mayer--Vietoris-sec}

In this section we recall the concept of a $k$-free idempotent left cover and the Mayer--Vietoris complex from \cite[Section 2]{Boyde2}.

\begin{defn}
Let $A$ be a $k$-algebra. Let $I$ be a two-sided ideal of $A$. Let $w\geqslant h \geqslant 1$. An \emph{idempotent left cover of $I$ of height $h$ and width $w$} is a collection of left ideals $J_1,\dotsc , J_w$ in $A$ such that
\begin{itemize}
    \item $J_1+\cdots +J_w=I$;
    \item for $S\subset \ul{w}$ with $\llv S\rrv \leqslant h$, the intersection
    \[\bigcap_{i\in S} J_i\]
    is either zero or is a principal left ideal generated by an idempotent.
\end{itemize}
If $I$ is free as a $k$-module, then an idempotent left cover is said to be \emph{$k$-free} if there is a choice of $k$-basis for $I$ such that each $J_i$ is free on a subset of this basis.
\end{defn}

\begin{defn}
Let $A$ be a $k$-algebra. Let $I\subset A$ be a two-sided ideal. Let $J_1,\dotsc , J_w$ be an idempotent left cover of $I$. The \emph{Mayer--Vietoris complex associated to the idempotent left cover}, $C_{\star}$, is the chain complex of left $A$-modules defined as follows. We set
\[C_p= \underset{\llv S \rrv = p}{\bigoplus_{S \subset \ul{w}}}\bigcap_{i\in S} J_i\]
for $1\leqslant p \leqslant w$. We set $C_0=A$, $C_{-1}=A/I$ and $C_n=0$ for $n>w$ and $n<-1$.

The differential $C_0\rightarrow C_{-1}$ is the projection map $A\rightarrow A/I$. The differential $C_1\rightarrow C_0$ is the direct sum of the inclusion of the left ideals $J_i\rightarrow A$. For $p\geqslant 2$, the differential $C_p\rightarrow C_{p-1}$ is defined on the summand $\cap_{i\in S} J_i$ by
\[x\mapsto \sum_{j\in S} (-1)^{\#(S,j)} i_{(S,j)}(x)  \]
where $\#(S,j)$ is the number of elements of $S$ that are less than $j$ and $i_{(S,j)}$ is the inclusion
\[\bigcap_{i\in S} J_i \rightarrow \bigcap_{i\in S\setminus \llb j\rrb} J_i.\]
\end{defn}

The Mayer--Vietoris complex satisfies the following important property (see \cite[Lemma 2.2]{Boyde2}).

\begin{lem}
\label{Boyde-lem}
The Mayer--Vietoris complex associated to a $k$-free idempotent left cover is acyclic. \fakeqed
\end{lem}

We will require the following lemma.

\begin{lem}
\label{Boyde-lem-2}
Let $X$ be a right $A$-module and let $Y$ be a left $A$-module. Let $J$ be a left ideal of $A$ which is generated by idempotents and acts as multiplication by $0$ on both $X$ and $Y$. Then
\[X\otimes_A J=0 \quad \text{and} \quad \Hom_A(J, Y)=0.\]
\end{lem}
\begin{proof}
    The statement on the tensor product is \cite[Lemma 2.3]{Boyde2}. The proof of the second part runs similarly. It suffices to consider an idempotent generator of $J$. Let $\alpha \in A$ and let $e$ be the idempotent generator of $J$. For $f \in \Hom_A(J, Y)$, we have $f(\alpha e)=f(\alpha e^2)=\alpha ef(e)=0$, where the first equality uses idempotence of $e$, the second uses $A$-linearity and the final equality uses the fact that $J$ acts trivially on $Y$.
\end{proof}

\section{Link states and technical lemmas}

\subsection{Link states}
\label{link-state-subsec}

We recall the notion of link state for dilute Temperley--Lieb diagrams and define families of ideals based on these. We refer the reader to \cite[Section 3]{BSA} for a more extensive discussion of link states for dilute Temperley--Lieb diagrams. Our definitions and notation follow that of Boyde \cite[Section 1]{Boyde}.

\begin{defn}
By slicing vertically down the middle of a dilute Temperley--Lieb $n$-diagram we obtain its \emph{left link state} and \emph{right link state}. Explicitly, we split all propagating edges at their midpoint and preserve all non-propagating edges. A propagating edge that has been split is called a \emph{defect} in its link state.
\end{defn}

\begin{rem}
The right link state of a dilute Temperley--Lieb $n$-diagram consists of a column of $n$ vertices, labelled $\ol{1},\dotsc \ol{n}$, such that at each vertex we have one of the following three situations:
\begin{itemize}
    \item the vertex has a hanging edge, called a \emph{defect};
    \item the vertex is connected to precisely one other vertex by a non-propagating edge;
    \item the vertex is an isolated vertex.
\end{itemize}
\end{rem}

\begin{eg}
Here is an example of a dilute Temperley--Lieb $6$-diagram (on the left) and its right link state (on the right).   

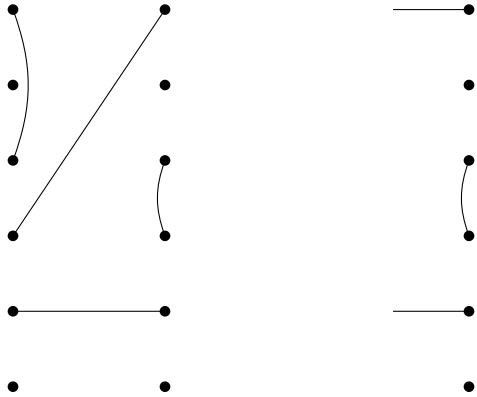
\begin{center}
\begin{tikzpicture}
\fill (7,0) circle[radius=2pt];
\fill (7,1) circle[radius=2pt];
\fill (7,2) circle[radius=2pt];
\fill (7,-1) circle[radius=2pt];
\fill (7,-2) circle[radius=2pt];
\fill (7,-3) circle[radius=2pt];
\path[-] (7,-1) edge [bend left=20] (7,0);
\draw (7,2) -- (6,2);
\draw (7,-2) -- (6,-2);

\fill (1,0) circle[radius=2pt];
\fill (1,1) circle[radius=2pt];
\fill (1,2) circle[radius=2pt];

\fill (3,0) circle[radius=2pt];
\fill (3,1) circle[radius=2pt];
\fill (3,2) circle[radius=2pt];

\fill (1,-1) circle[radius=2pt];
\fill (1,-2) circle[radius=2pt];
\fill (3,-1) circle[radius=2pt];
\fill (3,-2) circle[radius=2pt];
\fill (1,-3) circle[radius=2pt];
\fill (3,-3) circle[radius=2pt];
\path[-] (3,-1) edge [bend left=20] (3,0);
\draw (3,2) -- (1,-1);
\draw (1,-2) -- (3,-2);
\path[-] (1,0) edge [bend right=20] (1,2);

\end{tikzpicture}    
\end{center}
\end{eg}

\begin{defn}
If we have a right link state of a dilute Temperley--Lieb $n$-diagram we can remove two defects and replace them with a non-propagating edge joining the two vertices, so long as the resulting diagram is also a right link state of a dilute Temperley--Lieb $n$-diagram. This operation is called a \emph{splice}. For a right link state $p$, let $J_p$ denote the left ideal of $dTL_n(\delta)$ with basis given by the diagrams having right link state obtained from $p$ by a (possibly empty) sequences of splices.
\end{defn}

\begin{rem}
It follows from the definition of dilute Temperley--Lieb $n$-diagram that we can splice defects at vertices $\ol{i}$ and $\ol{k}$ if and only if there is not a defect at any vertex $\ol{j}$ with $i<j<k$.
\end{rem}

In order to prove our results it will be useful to consider \emph{double diagrams} and \emph{sesqui-diagrams} as introduced in \cite[Subsection 8.1]{Boyde}.

\begin{defn}
The composition of two dilute Temperley--Lieb diagrams $d_1$ and $d_2$ involved a diagram $d_1\ast d_2$ with three columns of vertices, formed by identifying the right-hand column of vertices of $d_1$ with the left-hand column of vertices in $d_2$. We call such a diagram a \emph{double diagram}. Henceforth, when using double diagrams we will label the left-hand column of vertices by $1,\dotsc , n$, the middle column of vertices with $1^{\prime},\dotsc , n^{\prime}$ and the right-hand column of vertices with $\ol{1},\dotsc , \ol{n}$.     
\end{defn}

\begin{defn}
Let $p$ be a right link state of a dilute Temperley--Lieb $n$-diagram. Let $d$ be a dilute Temperley--Lieb $n$-diagram. We define the \emph{sesqui-diagram} $(p,d)$ to be the diagram formed by identifying the vertices $\ol{1},\dotsc , \ol{n}$ in $p$ with the vertices $1,\dotsc , n$ in $d$. After performing this identification, we relabel the vertices. The left-hand column of vertices will be labelled by $1^{\prime},\dotsc ,n^{\prime}$ and we preserve the labels $\ol{1},\dotsc , \ol{n}$ on the right-hand column.  
\end{defn}

\subsection{Technical lemmas}
Boyde proves two lemmas (\cite[Lemmas 8.6 and 8.12]{Boyde}) which allow him to identify the necessary idempotents to prove parameter-independent results for the Temperley--Lieb algebras. In this subsection, we prove the analogues of these lemmas for the dilute Temperley--Lieb algebras. Suppose we have a right link state, $p$. Lemma \ref{dTL-idempotent-lem}, below, provides conditions under which a diagram $e$ will satisfy $ye=y$ for all $y\in J_p$. In other words, if there exists such a diagram $e$, it is a principal idempotent generator of $J_p$. Following that, Lemma \ref{idem-lem-2} shows that such a diagram always exists if $p$ has at least one defect. We note that three of our conditions in Lemma \ref{dTL-idempotent-lem} (Conditions \ref{C1}, \ref{C3} and \ref{C4}) correspond precisely to the conditions given by Boyde (Lemma 8.6). Our extra condition ensures compatibility with the isolated vertices which are allowed in dilute Temperley--Lieb diagrams.

\begin{lem}
\label{dTL-idempotent-lem}
Let $p$ be a right link state of a dilute Temperley--Lieb $n$-diagram. Suppose that $e$ is a diagram in $dTL_n(\delta)$ such that
\begin{enumerate}
    \item\label{C1} $e$ has right link state $p$,
    \item\label{C2} if the vertex $\ol{j}$ is isolated in $p$ then the vertices $j$ and $\ol{j}$ are isolated in $e$,
    \item\label{C3} if $p$ has a defect at vertex $\ol{j}$ then there is a sequence of edges in the sesqui-diagram $(p,e)$ that connects the vertices $j^{\prime}$ and $\ol{j}$,
    \item\label{C4} in the sesqui-diagram $(p,e)$, every non-propagating edge in the left-hand column appears in precisely one of the sequences of edges arising in Condition \ref{C3}.
\end{enumerate}
Then for any $y\in J_p$, we have $ye=y$.
\end{lem}
\begin{proof}
We will prove the following equivalent statement. We will show that if two vertices in the set $\llb 1,\ol{1},\dotsc , n,\ol{n}\rrb$ are connected in $y$, then they are also connected in the double diagram $y\ast e$ and that the double diagram $y\ast e$ contains no loops, no floating edges and no edges lying entirely within the middle column. We break the proof up into parts.

\begin{enumerate}
\item In the first part, we show that if two vertices in the set $\llb 1,\ol{1},\dotsc , n,\ol{n}\rrb$ are connected in $y$, then they are also connected in the double diagram $y\ast e$.

We start by observing that any non-propagating edge in the left hand column of $y$ is present in the double diagram $y \ast e$. Suppose $y$ has a non-propagating edge in the right-hand column between vertices $\ol{i}$ and $\ol{j}$. There are two cases. Since $y\in J_p$, its right link state is formed from $p$ by a (possibly empty) number of splices. Therefore, a non-propagating edge in the right hand column of $y$ is either also a non-propagating edge in $p$ or it is the result of a splice. If the edge exists in $p$, then the non-propagating edge exists in the right-hand column of the double diagram $y\ast e$ since $e$ has right link state $p$ by Condition \ref{C1}. If the non-propagating edge in $y$ is the result of splicing two defects in $p$, then Condition \ref{C3} tells us that there is a sequence of edges joining the vertex $\ol{i}$ in the right column of the double diagram to the vertex $i^{\prime}$. This is joined to vertex $j^{\prime}$ in the double diagram by the non-propagating edge in $y$ and Condition \ref{C3} tells us that vertex $j^{\prime}$ is connected to the vertex $\ol{j}$ in the right-hand column of $y\ast e$. Therefore the vertices $\ol{i}$ and $\ol{j}$ are connected in the right-hand column of the double diagram $y\ast e$.

To finish, suppose $y$ has a propagating edge from $i$ to $\ol{k}$, so that in the double diagram $y\ast e$ there is an edge from the vertex $i$ to the vertex $k^{\prime}$. Therefore, in the right link state of $y$ there is a defect at $\ol{k}$. Since $y \in J_p$, this means that there is a defect at $\ol{k}$ in $p$ as well. Condition \ref{C3} tells us that there is a sequence of edges in the sesqui-diagram $(p,e)$ joining the vertices $k^{\prime}$ and $\ol{k}$ and therefore the vertex $i$ in the left-hand column of the double diagram $y\ast e$ is connected to $\ol{k}$ in the right-hand column.

We have shown that all vertices that were connected in $y$ are connected in the double diagram $y\ast e$.

\item Condition \ref{C2} tells us that a vertex is isolated in $p$ if and only if it isolated in the right-hand column of $e$ if and only if it is isolated in the middle column of the double diagram $y \ast e$. Therefore the isolated vertices in the left-hand column of $y$ are precisely the isolated vertices in the left-hand column of $y\ast e$ and the isolated vertices in the right-hand column of $y$ are precisely the isolated vertices in the right-hand column of $y\ast e$.

\item We now show that there can be no floating edges. We have already shown that any propagating edge in $y$ is connected to the right-hand column of the double diagram $y\ast e$ and therefore cannot be a floating edge in the double diagram. The only other possibility is a floating edge arising from the diagram $e$. Condition \ref{C3} tells us that if $e$ has a propagating edge terminating at $\ol{j}$, then there is a sequence of edges connecting the vertices $j^{\prime}$ and $\ol{j}$ in the sesqui-diagram $(p,e)$. Since $p$ is the right link state of $e$ it also has a defect at $j$ and since $y\in J_p$, it follows that every propagating edge in $e$ is part of a sequence of edges in $y\ast e$ that connect the left-hand column to the right-hand column or connect the right-hand column to the right-hand column. Therefore we can have no floating edges.

\item We now show that there can be no loops. A loop in the double diagram $y\ast e$ must be formed by a non-zero number of non-propagating edges in the right-hand column of $y$ and a non-zero number of non-propagating edges in the left-hand column of $e$. However, Condition \ref{C4} tells us that every non-propagating edge in the left-hand column of $e$ must be part of a sequence of edges connecting a defect in $(p,e)$ at some vertex $j^{\prime}$ to the vertex $\ol{j}$. Since $y\in J_p$, in the double diagram $y\ast e$ this sequence of edges must either connect to the left-hand column or the right-hand column and so no loops can be formed.

\item Finally, we show that there are no edges lying entirely within the middle column. We have already dealt with non-propagating edges in the left-hand column of $e$ in the previous point. We have also shown that any non-propagating edge in the right-hand column of $y$ arising from a splice of two defects in its right link state must have both vertices connected to the right-hand column of the double diagram $y\ast e$, so cannot lie entirely within the middle column. Finally, if we have a non-propagating edge in the right-hand column of $y$ which also exists in the link state $p$, Conditions \ref{C3} and \ref{C4}, together with the fact $y \in J_p$ tells us that this edge must be connected to the right-hand column of the double diagram $y\ast e$ and therefore cannot lie entirely within the middle column.
\end{enumerate}
Therefore, for any $y\in J_p$, $ye=y$ as required.
\end{proof}

\begin{lem}
\label{idem-lem-2}
Let $p$ be a right link state of a dilute Temperley--Lieb $n$-diagram with at least one defect. There exists a diagram $e_p\in dTL_n(\delta)$ satisfying the conditions of Lemma \ref{dTL-idempotent-lem}.    
\end{lem}
\begin{proof}
We start with our link state $p$. Since $p$ has at least one defect, it can have at most $n-1$ isolated vertices. Suppose we have $x$ isolated vertices with $0\leqslant x \leqslant n-1$. We temporarily forget the isolated vertices. The remaining vertices and edges form the right link state of Temperley--Lieb $(n-x)$-diagram with at least one defect. Applying \cite[Lemma 8.12]{Boyde} yields a Temperley--Lieb $(n-x)$-diagram satisfying Conditions \ref{C1}, \ref{C3} and \ref{C4} of Lemma \ref{dTL-idempotent-lem}. Replacing the $x$ isolated vertices in the right-hand column and inserting $x$ isolated vertices in the left-hand column symmetrically yields a dilute Temperley--Lieb $n$-diagram which also satisfies Condition \ref{C2} of Lemma \ref{dTL-idempotent-lem}.
\end{proof}

\section{An idempotent cover: proving Theorem \ref{main-thm}}
\label{main-thm-sec}
In this section we provide an idempotent left cover of the two-sided ideal $I$ spanned $k$-linearly by all dilute Temperley--Lieb $n$-diagrams having fewer than $n$ propagating edges and use it to prove Theorem \ref{main-thm}.

\begin{defn}
For each non-empty subset $S\subseteq \lbrace \ol{1},\dotsc , \ol{n}\rbrace$, let $K_S$ be the left ideal of $dTL_n(\delta)$ spanned $k$-linearly by dilute Temperley--Lieb $n$-diagrams such that the isolated vertices in the right-hand column correspond precisely to the elements of $S$.    
\end{defn}

\begin{defn}
For $1\leqslant i \leqslant n-1$, let $L_i$ denote the left ideal of $dTL_n(\delta)$ spanned $k$-linearly by dilute Temperley--Lieb $n$-diagrams such that there are no isolated vertices in the right-hand column and the vertices $\ol{i}$  and $\ol{i+1}$ are connected by a non-propagating edge.  
\end{defn}

\begin{lem}
\label{cover-lem}
The collection of left ideals consisting of $K_S$ for all non-empty subsets $S$ of $\lbrace \ol{1},\dotsc , \ol{n}\rbrace$ and $L_i$ for $1\leqslant i \leqslant n-1$ cover the two-sided ideal $I$.     
\end{lem}
\begin{proof}
A basis diagram in any $K_S$ must have at least one isolated vertex. Therefore this diagram cannot have $n$ propagating edges and so $K_S\subset I$. A basis diagram of $L_i$ must have at least one non-propagating edge and therefore cannot have $n$ propagating edges. Therefore $L_i\subset I$.

Conversely, a basis element of $I$ must either have isolated vertices in the right-hand column or, if it has no isolated vertices in the right-hand column, it must have at least one non-propagating edge in the right-hand column. If the diagram contains isolated vertices in the right-hand column then it lies in some $K_S$. Now suppose that the basis diagram has no isolated vertices in the right-hand column but has at least one non-propagating edge. The planarity condition on diagrams ensures that the diagram has a non-propagating edge between two consecutive vertices in the right-hand column. 
\end{proof}

\begin{lem}
\label{zero-intersection-lem}
The following intersections of ideals are zero.
\begin{enumerate}
    \item\label{mix-int} An intersection involving an ideal of the form $K_S$ and an ideal of the form $L_i$ is zero.
    \item\label{K-int} Let $S$ and $T$ be two distinct, non-empty subsets of $\lbrace \ol{1},\dotsc , \ol{n}\rbrace$. The intersection $K_S\cap K_T$ is zero.
    \item\label{L-int} Let $U\subseteq \ul{n-1}$. We have
\[\bigcap_{i\in U} L_i= 0\]
if and only if $U$ contains consecutive elements of $\ul{n-1}$. 
\end{enumerate}
\end{lem}
\begin{proof}
For Part \ref{mix-int}, a diagram in the intersection would need to have a non-zero number of isolated vertices in the right-hand column (since it lies in some $K_S$) but would also need to have no isolated vertices in the right-hand column (since it lies in some $L_i$). This is clearly a contradiction, so such an intersection is empty.      

For Part \ref{K-int}, the definitions of $K_S$ and $K_T$ dictate precisely which vertices in the right-hand column are isolated. If the two sets $S$ and $T$ are distinct, there can be no intersection.

For Part \ref{L-int}, on the one hand, if $U$ contains two consecutive elements of $\ul{n-1}$, say $j$ and $j+1$, then a diagram in the intersection must have a non-propagating edge between $\ol{j}$ and $\ol{j+1}$ and also a non-propagating edge between $\ol{j+1}$ and $\ol{j+2}$. In other words, such a diagram would have two edges incident to vertex $\ol{j+1}$ which contradicts the definition of a dilute Temperley--Lieb diagram.   

On the other hand, suppose $U$ contains no consecutive elements. Let $d$ be the diagram such that
\begin{itemize}
    \item for each $i\in U$, $d$ has a non-propagating edge from $i$ to $i+1$ and a non-propagating edge from $\ol{i}$ to $\ol{i+1}$ and
    \item for each vertex $j$ in the left-hand column which is not part of such a non-propagating edge, there is a propagating edge between $j$ and $\ol{j}$.
\end{itemize}
This lies in the intersection and so the intersection is non-zero.
\end{proof}

\begin{lem}
\label{K-idem-lem}
For each non-empty subset $S\subseteq \lbrace \ol{1},\dotsc , \ol{n}\rbrace$, the ideal $K_S$ is principal and generated by an idempotent. Hence each $K_S$ is projective as a left $dTL_n(\delta)$-module.
\end{lem}
\begin{proof}
Suppose $S=\lbrace \ol{1},\dotsc , \ol{n}\rbrace$. Let $q$ be the right link state consisting of $n$ isolated vertices. Then $K_S=J_q$, the left ideal spanned $k$-linearly by all dilute Temperley--Lieb $n$-diagrams whose right link state is $q$. This ideal is generated by the dilute Temperley--Lieb $n$-diagram with no edges whatsoever and one observes that this diagram is idempotent.  

Now suppose that $S$ is a non-empty strict subset of $\lbrace \ol{1},\dotsc , \ol{n}\rbrace$. Let $q_1$ denote the right link state such that the vertices labelled by elements of $S$ are isolated and vertices labelled by elements the complement of $S$ have defects. 

We have $K_S=J_{q_1}$, the left ideal spanned $k$-linearly by all dilute Temperley--Lieb $n$-diagrams whose right link state can be obtained from $q_1$ by a valid sequence of splices. Since $S$ is non-empty and $S\neq \lbrace \ol{1},\dotsc , \ol{n}\rbrace$, the right link state $q_1$ contains at least one defect. By combining Lemmas \ref{dTL-idempotent-lem} and \ref{idem-lem-2}, there exists an element $e_{q_1}$ such that right multiplication by $e_{q_1}$ gives a retraction $dTL_n(\delta)\rightarrow J_{q_1}$. Therefore, by \cite[Lemma 2.5]{Boyde2}, $K_S=J_{q_1}$ is principal and generated by an idempotent. 
\end{proof}

\begin{lem}
\label{odd-lem}
Let $n$ be odd. Let $U\subset \ul{n-1}$ be a subset containing no consecutive elements. Then
\[\bigcap_{i\in U} L_i\]
is principal and generated by an idempotent. Hence each such intersection is projective as a left $dTL_n(\delta)$-module.
\end{lem}
\begin{proof}
Let $q_2$ denote the right link state such that there is a non-propagating edge between $\ol{i}$ and $\ol{i+1}$ for each $i\in U$ and defects at all other vertices.

We have 
\[\bigcap_{i\in U} L_i=J_{q_2},\]
the left ideal spanned $k$-linearly by all dilute Temperley--Lieb $n$-diagrams whose right link state can be obtained from $q_2$ by a valid sequence of splices. Since $n$ is odd we must have at least one defect. 

By combining Lemmas \ref{dTL-idempotent-lem} and \ref{idem-lem-2}, there exists an element $e_{q_2}$ such that right multiplication by $e_{q_2}$ gives a retraction $dTL_n(\delta)\rightarrow J_{q_2}$. Therefore, by \cite[Lemma 2.5]{Boyde2}, $\bigcap_{i\in U} L_i=J_{q_2}$ is principal and generated by an idempotent. 
\end{proof}

\begin{lem}
\label{even-lem}
Let $n$ be even. Let $U\subset \ul{n-1}$ be a subset containing no consecutive elements such that $U\neq \lbrace 1,3,\dotsc , n-1\rbrace$. Then
\[\bigcap_{i\in U} L_i\]
is principal and generated by an idempotent. Hence each such intersection is projective as a left $dTL_n(\delta)$-module.   
\end{lem}
\begin{proof}
The same proof as Lemma \ref{odd-lem} applies here since we must have at least one defect.  
\end{proof}

The following is the analogue of the cup module appearing in \cite[Theorem B]{Sroka}.

\begin{defn}
\label{cup-mod-defn}
Let $n$ be even. Let $\cupmod(n)$ denote the left $dTL_n(\delta)$-module spanned by all dilute Temperley--Lieb $n$-diagrams with a non-propagating edge between $\ol{i}$ and $\ol{i+1}$ for each $i\in \lbrace 1,3,\dotsc, n-1\rbrace$.
\end{defn}

It is immediate for $n$ even and $U=\llb 1,3,\dotsc, n-1\rrb$ that there is an isomorphism of left $dTL_n(\delta)$-modules
\[\cupmod(n) \cong \bigcap_{i\in U} L_i.\]

However, the next lemma shows that $\cupmod(n)$ is also isomorphic to another of our ideals of the form $K_S$ and is therefore projective. This can be seen as the main reason why the (co)homology of the dilute Temperley--Lieb algebras is different from that of the classical Temperley--Lieb algebras; the cup module appearing in \cite[Theorem B]{Sroka} is not projective in general, as can be seen from \cite[Theorem C]{BH1} or the results of \cite{BBRWS,Boyde-planar}.

\begin{lem}
\label{cup-proj-lem}
Let $n$ be even and let $S = \lbrace \ol{1},\dotsc , \ol{n}\rbrace$. There is an isomorphism of left $dTL_n(\delta)$-modules
\[K_S\cong \cupmod(n).\]
In particular, $\cupmod(n)$ is projective as a left $dTL_n(\delta)$-module and  \[\1\otimes_{dTL_n(\delta)} \cupmod(n)= \Hom_{dTL_n(\delta)}(\cupmod(n),\1)=0.\] 
\end{lem}
\begin{proof}
We begin by noting that basis diagrams in both $K_S$ and $\cupmod(n)$ can have no propagating edges.

Let $d_{cup}$ be the dilute Temperley--Lieb $n$-diagram whose left link state consists solely of isolated vertices and whose right link state has a non-propagating edge between $\ol{i}$ and $\ol{i+1}$ for each $i\in \lbrace 1,3,5,\dotsc , n-1\rbrace$. 

Right multiplication by $d_{cup}$ gives a map of left $dTL_n(\delta)$-modules $K_S\rightarrow \cupmod(n)$.

Now consider a basis diagram, $d$, in $\cupmod(n)$. This can be written as a product $d=d_l d_r$: we take $d_l$ to be the dilute Temperley--Lieb $n$-diagram whose left link state is the left link state of $d$ and whose right link state consists solely of isolated vertices; we take $d_r$ to be the dilute Temperley--Lieb $n$-diagram whose left link state consists solely of isolated vertices and whose right link state has a non-propagating edge between $\ol{i}$ and $\ol{i+1}$ for each $i\in \lbrace 1,3,5,\dotsc , n-1\rbrace$. 

We note that $d_l$ is a basis diagram in $K_S$ and therefore the assignment $d\mapsto d_{l}$ determines a map of left $dTL_n(\delta)$-modules $\cupmod(n)\rightarrow K_S$.

These maps are mutually inverse, giving the required isomorphism. The remainder of the statement now follows from Lemmas \ref{K-idem-lem} and \ref{Boyde-lem-2}.   
\end{proof}

We can now prove Theorem \ref{main-thm}.

\begin{proof}[Proof of Theorem \ref{main-thm}]
Recall that $I\subset dTL_n(\delta)$ is the two-sided ideal spanned by all dilute Temperley--Lieb $n$-diagrams having fewer than $n$ propagating edges.

We consider the Mayer--Vietoris complex, $C_{\star}$, of the cover of $I$ provided by Lemma \ref{cover-lem}, which consists of $K_S$ for all non-empty subsets $S$ of $\lbrace \ol{1},\dotsc , \ol{n}\rbrace$ and $L_i$ for $1\leqslant i \leqslant n-1$. We split into three cases: $n$ odd, $n=2$ and $n$ even with $n>2$. By Lemma \ref{zero-intersection-lem} any intersection of the form $K_S\cap L_i$ is zero, as is any intersection of the form $K_S \cap K_T$ for distinct subsets $S$ and $T$ of $\lbrace \ol{1},\dotsc , \ol{n}\rbrace$, as is any intersection of the form $\cap_{i\in U} L_i$ if and only if $U\subset \ul{n-1}$ contains consecutive elements.  

Therefore, the Mayer--Vietoris can take the following forms. Let $1\leqslant x\leqslant \lfloor\frac{n}{2}\rfloor -1$. To ease notation, let
\[\mathbf{K}_{n} = \bigoplus_{S} K_S  \quad \text{and} \quad \mathbf{L}_x= \underset{\lvert U\rvert=x}{\bigoplus_{U\subset \ul{n-1}}} \bigcap_{i\in U} L_i\]
where each $S$ is a non-empty subset of $\lbrace \ol{1},\dotsc , \ol{n}\rbrace$, and each set $U$ contains no consecutive elements of $\ul{n-1}$.

When $n=2$ we have
\[0\leftarrow \1 \leftarrow dTL_2(\delta)\leftarrow \mathbf{K}_2\oplus \cupmod(2)\leftarrow 0,\]
with $\1$ in degree $-1$ and $\mathbf{K}_2\oplus \cupmod(2)$ appearing in degree $1$.

For $n>2$ even we have
\[0\leftarrow \1 \leftarrow dTL_n(\delta)\leftarrow \mathbf{K}_n\oplus \mathbf{L}_1\leftarrow \mathbf{L}_2\leftarrow \cdots \leftarrow \mathbf{L}_{\frac{n}{2}-1} \leftarrow  \cupmod(n)\leftarrow 0,\]
with $\1$ appearing in degree $-1$ and $\cupmod(n)$ appearing in degree $\frac{n}{2}$.

For $n$ odd, we have
\[0\leftarrow \1 \leftarrow dTL_n(\delta)\leftarrow \mathbf{K}_n \oplus \mathbf{L}_1\leftarrow \mathbf{L}_2 \leftarrow\cdots \leftarrow \mathbf{L}_{\lfloor \frac{n}{2}\rfloor -1 } \leftarrow 0\]
with $\1$ appearing in degree $-1$ and $\mathbf{L}_{\lfloor \frac{n}{2}\rfloor -1 }$ appearing in degree $\lfloor\frac{n}{2}\rfloor -1$.

In all cases, Lemma \ref{Boyde-lem} tells us that the Mayer--Vietoris complex is acyclic. Lemmas \ref{K-idem-lem}, \ref{odd-lem}, \ref{even-lem} and \ref{cup-proj-lem} tell us that each ideal of the form $K_S$, $\cap_{i\in U} L_i$ or $\cupmod(n)$ is projective as a left $dTL_n(\delta)$-module. Hence each term of the Mayer--Vietoris above degree $-1$ is a projective left $dTL_n(\delta)$-module.

In other words, in each case, the complex $C_{\geqslant 0}$ is a projective resolution of $\1$ by left $dTL_n(\delta)$-modules. 

Therefore, we have
\[\Tor_{\star}^{dTL_n(\delta)}(\1,\1)= H_{\star}(\1 \otimes_{dTL_n(\delta)} C_{\geqslant 0}) \quad \text{and} \quad \Ext_{dTL_n(\delta)}^{\star}(\1,\1)=H^{\star}(\Hom_{dTL_n(\delta)}(C_{\geqslant 0},\1)).\]

However, by the above lemmas together with Lemma \ref{Boyde-lem-2}, we have 
\begin{itemize}
\item $\1\otimes_{dTL_n(\delta)} C_0=\1\otimes_{dTL_n(\delta)}dTL_n(\delta)\cong k$;
\item $\1\otimes_{dTL_n(\delta)} C_{p}=0$ for $p>0$;
\item $\Hom_{dTL_n(\delta)}(C_0,\1)= \Hom_{dTL_n(\delta)}(dTL_n(\delta) ,\1)\cong k$ and
\item $\Hom_{dTL_n(\delta)}(C_{p},\1)=0$ for $p>0$
\end{itemize}
from which we deduce the result.   
\end{proof}

\bibliographystyle{alpha}
\bibliography{dTL-refs}

\end{document}